\documentclass[10pt, fleq]{amsart}
\usepackage{amssymb,amsmath,latexsym,amsbsy,color,enumerate}
\numberwithin{equation}{section}

\def\RR{\mbox{$I\hspace{-.06in}R$}}
\def\CC{\mbox{$C\hspace{-.11in}\protect\raisebox{.5ex}{\tiny$/$}
\hspace{.06in}$}}

\newtheorem{theorem}{Theorem}

\newtheorem{lemma}{Lemma}
\newtheorem{remark}{Remark}
\newtheorem{proposition}{Proposition}

\begin{document}
 \title{The growth at infinity of a sequence of entire functions of bounded orders}
\author{Dang Duc Trong}
    \address{Department of Mathematics, University of natural sciences, Hochiminh city,
    Vietnam}
 \email{ddtrong@mathdep.hcmuns.edu.vn}
\author{Truong Trung Tuyen}
    \address{Department of Mathematics, Indiana University Bloomington, IN 47405 USA}
 \email{truongt@indiana.edu}
\thanks{}
    \date{\today}
    \begin{abstract}
In this paper we shall consider the growth at infinity of a sequence $(P_n)$ of entire functions of bounded orders. Our results extend the results in
\cite{trong-tuyen2} for the growth of entire functions of genus zero. Given a sequence of entire functions of bounded orders $P_n(z)$, we found a nearly
optimal condition, given in terms of zeros of $P_n$, for which $(k_n)$ that we have
\begin{eqnarray*}
\limsup _{n\rightarrow\infty}|P_n(z)|^{1/k_n}\leq 1
\end{eqnarray*}
for all $z\in \mathbb C$ (see Theorem \ref{theo5}). Exploring the growth of a sequence of entire functions of bounded orders lead naturally to an
extremal function which is similar to the Siciak's extremal function (See Section 6).
\end{abstract}
\maketitle
\section{Introduction and main results.}       
Let $P(z):\mathbb C\rightarrow\mathbb C$ be an entire function. We recall that (see Lecture 1 in \cite{lev}): if
$$P(z)=\sum _{n=0}^{\infty}a_nz^n$$
then its order $\rho$ is
$$\rho =\limsup_{n\rightarrow\infty}\frac{n\log n}{\log (1/|a_n|)}.$$
An entire function is called of genus zero if its order is less than $1$.

If $P(z)$ is an entire function of finite order $\rho$, by Hadamard factorization theorem (see Theorem 1 page 26 in \cite{lev}), $P(z)$ has a
representation
\begin{equation}
P(z)=az^me^{W_q(z)}\prod _{j}G(\frac{z}{z_j},p)\label{Sec8.1}
\end{equation}
where $p=[\rho ]$ the integer part of $\rho$, $W_q(z)$ is a polynomial of degree $q\leq \rho$ with $W_q(0)=0$, $z_j's$ are zeros of $P(z)$ different from
$0$, and
\begin{eqnarray*}
G(z,p)=(1-z)\exp \{z+\frac{z^2}{2}+\ldots +\frac{z^p}{p}\}.
\end{eqnarray*}
For an entire function $P(z)$ of order $\rho$ with representation (\ref{Sec8.1}) we define its degree $d^*(P)$ by
\begin{equation}
d^*(P)=m+\sup _{|z|\leq 1}|W_q (z)|+\sum _{|z_j|\leq 1}\frac{1}{|z_j|^p}+\sum _{|z_j|> 1}\frac{1}{|z_j|^{p+1}}.\label{Sec8.2}
\end{equation}

Remark: If $P(z)$ is an entire function of genus zero then in the representation (\ref{Sec8.1}) we have $q=p=0$. If $P(z)$ is a polynomial then
$d^*(P)\leq deg(P)$ where $deg(P)$ is the usual degree of $P(z)$.

The growth at infinity of entire functions is a topic of great concernment. Let $(P_n)$ be a sequence of entire functions and let $(k_n)$ be a sequence
of positive numbers. It is an interesting question that for which sequences $(P_n)$ and $(k_n)$  we have
\begin{eqnarray*}
\limsup _{n\rightarrow\infty}|P_n(z)|^{1/k_n}\leq 1
\end{eqnarray*}
for all $z\in \mathbb C$. This question arises naturally as we want to show that a series
\begin{equation}
s(z)=\sum _{n=1}^{\infty}P_n(z)\label{series}
\end{equation}
converges locally uniformly to an entire function. Usually we want to find a sequence of positive numbers $k_n\geq n$ such that
\begin{eqnarray*}
\limsup _{n\rightarrow\infty}|P_n(z)|^{1/k_n}<1
\end{eqnarray*}
for all $z\in \mathbb C$. Of course if the series (\ref{series}) converges then for any sequence $(k_n)$ that diverges to $\infty$ we have
\begin{equation}
\limsup _{n\rightarrow\infty}|P_n(z)|^{1/k_n}\leq 1\label{GeomtricRate}
\end{equation}
for all $z\in \mathbb C$. If we choose the sequence $(k_n)$ very large then it is mostly likely that the LHS of (\ref{GeomtricRate}) equals to $1$ so we
can not conclude that (\ref{series}) converges, while if we choose $(k_n)$ very small then we may not know whether the LHS of (\ref{GeomtricRate}) is
bounded or not (in particular when we want to show that if $s(z)$ converges on a certain subset of $\mathbb C$ then it converges on $\mathbb C$). Hence
it is worthwhile to consider what conditions that any "generic" sequence $(k_n)$ must satisfy if it satisfies (\ref{GeomtricRate}). Here "generic" means
that the constant $C_0^*$ defined in Section 2 is not zero. (For the case $C_0^*=0$ we can not hope to have any conclusion about the sequence $(k_n)$,
because if $P_n(z)$ converges uniformly to $0$ then any bounded sequence $k_n$ will satisfy (\ref{GeomtricRate})). We state this question as

Question 1: Given a sequence of entire functions $(P_n)$. Assume that $(k_n)$ is a sequence of positive numbers such that (\ref{GeomtricRate}) is
satisfied. How can we say about $k_n$ in terms of zeros of $P_n$'s?

In \cite{mul-yav}, the authors gained interested results which combine the growth of a sequence of polynomials on a "small subset" of $\Bbb C$ with the
growth of itself on the whole plane. The "small subsets" as mentioned are non-thin. We recall that (see \cite{mul-yav}) a domain $G$, with $\partial G$
having positive capacity, is non-thin at an its boundary point $\zeta \in
\partial G$ (or $\zeta$ is a regular point of $G$) if and only if
\begin{eqnarray*}
\lim _{z\in G, z\rightarrow \zeta}g(z,w)=0
\end{eqnarray*}
for all $w\in G$ where $g(.,.)$ is the Green function of $G$.
One of the main results in \cite{mul-yav} is stated as below (see Lemma 2 in \cite{mul-yav})
\begin{proposition}
Let $(d_n)$ be a sequence of positive numbers and let $(P_n)$ be a sequence of polynomials satisfying $deg(P_n)\leq d_n$. If $E\subseteq \Bbb C$ is
closed and non-thin at $\infty$ so that
$$\limsup _{n\rightarrow\infty}|P_n(z)|^{1/d_n}\leq 1,~\mbox{ for all }z\in E,$$
then
$$\limsup _{n\rightarrow\infty}||P_n||_R^{1/d_n}\leq 1,~\mbox{ for all }R>0,$$
where $||P_n||_R=\sup \{|P_n(z)|:~|z|\leq R \} .$
\label{lemmulyav}\end{proposition}

Saying roughly, Proposition \ref{lemmulyav} states that if $E$ is a set being non-thin at infinity and $(P_n)$ is a sequence of polynomials having a
geometric rate growth at each point in $E$ then $(P_n)$ has a geometric rate growth in $\Bbb C$ also. Here by geometric rate growth we refer to one
property that the sequence of polynomials in Proposition \ref{lemmulyav} satisfies: if moreover we have that $(d_n)$ is an increasing sequence of integer
numbers, and
\begin{eqnarray*}
\limsup _{n\rightarrow\infty}||P_n||^{1/d_n}_K<1
\end{eqnarray*}
for any compact set $K\subset \Bbb C$, then the summation $\sum_{n=1}^{\infty} P_n(z)$ converges with locally geometric rate. Hence for the sake of
abbreviation, we use the term geometric rate growth to call any sequence of entire functions $(P_n)$ satisfying the condition $\limsup
_{R\rightarrow\infty}||P_n||_{R}^{1/d_n}\leq 1$ for some sequence $(d_n>0)$ and for any $R>0$. (We can see easily why the upper bound constant $1$ is
important here, see \cite{wal} for detail.)

There are two interesting questions rising from this result

Question 2: Does the conclusion of Proposition \ref{lemmulyav} still hold if $P_n$ are non-polynomial entire functions?

Question 3: If $(P_n)$ has a non-geometric rate growth in $E$, i.e., if instead of the condition
\begin{eqnarray*}
\limsup _{n\rightarrow\infty}|P_n(z)|^{1/d_n}\leq 1,~\mbox{for all }z\in E,
\end{eqnarray*}
we requires only that
\begin{eqnarray*}
\limsup _{n\rightarrow\infty}|P_n(z)|^{1/d_n}\leq h(|z|),~\mbox{for all }z\in E,
\end{eqnarray*}
where $h$ is not necessary a bounded function, does $(P_n)$ remain the same rate of growth in $\Bbb C$?
\begin{remark}
The sequences $(P_n)$ whose growth is non-geometric arises naturally in practice. For a simple example, we can consider the case when $P_n(z)=z^n$ and
$d_n=n$. This sequence does not have the geometric rate growth, however $|P_n(z)|^{1/d_n}$ grows like $|z|$.
\end{remark}

In this paper we give some partial answers to these three questions when the sequence $(P_n)$ is of bounded orders: that is $\rho _n\leq \rho$ for all
$n\in \mathbb N$ where $\rho _n$ is the order of $P_n$, and $\rho >0$ is a constant. We obtained the more satisfying results for the case $(P_n)$ is a
sequence of entire function of genus zero (see Sections 2, 3, 4, and 5), and less satisfying results for the general case when $P_n$ are of bounded
orders (see Section 6).

For Question 1 we obtained a nearly optimal answer (see Theorem \ref{theo5}).

For Question 3, the answer is confirmation in the case $P_n(z)$ is of order zero, $h(z)$ is a polynomial and $E$ is a closed set satisfying
\begin{eqnarray*}
\limsup _{R\rightarrow\infty}\frac{\log cap(E_R)}{\log R}=\beta >0,
\end{eqnarray*}
where
$$E_{R}=E\cap \{z:|z|\leq R\}.$$
(See Theorem \ref{theo2}). As a corollary we immediately get that $E$ must be non-thin at infinity. This result was proved in \cite{mul-yav} using
Wiener's criterion.

The main tools that we used to get our results are: the degree of entire functions of finite orders; an equivalent relation between two different kinds
of growth, that is the modulus growth and the logarithmic-integration growth, of a sequence of entire functions of genus zero (see Lemmas \ref{lem2} and
\ref{lem3}); the Weierstrass inequality (see Lemma \ref{lem5}); and the integral representation of the Green's function (see Theorem \ref{theo2}).

Exploring the growth of a sequence of entire functions of bounded orders naturally lead us to consider an extremal function which is analogous to the
Siciak's extremal function (See Section 6). In \cite{bos-brudnyi-levenberg-totik} the authors also defined a similar extremal function for other classes
of functions. They also define a "degree" for a continuous functions $P(z)$ but based on a filtration of the space of continuous functions rather than
based on the zeros of the function $P(z)$ as our treatment here. Their degree is always an integer while our degree maybe any positive number.

In a recent paper of the second author, some of the results in this paper was extended to the case of entire functions of several complex variables (see
\cite{tuyen}).

This paper consists of six sections. In Sections 2 and 3 we consider the growth of a sequence of entire functions of genus zero. In Sections 4 and 5 we
consider some consequences, refinements, and examples (including grouped power series, grouped Fourier series, and Fourier transform). In Section 6 we
discuss the growth of a sequence of entire functions of bounded orders, and define an extremal function that is naturally related to the problem.
\section{Notations and Lemmas}

For an entire function $f$ we use notations
\begin{eqnarray*}
||f||_R&=&\sup _{\{|z|\leq R\}}|f(z)|\\
C(f,R)&=&\exp \{\frac{1}{2\pi}\int _{0}^{2\pi}\log|f(Re^{it})| \},\\
\eta (f,R)&=&\mbox{ the number of elements of } \{z:0<|z|\leq R,~f(z)=0\}.
\end{eqnarray*}

Hereafter (in particular, in Sections 2, 3 and 4), unless specified otherwise, we always consider a sequence of positive numbers $(k_n)$ and a sequence
of entire functions of genus zero $(P_n)$ having the following form
\begin{equation}
P_n(z)=a_nz^{\alpha _n}\prod _{j}(1-z/z_{n,j}).\label{RepresentationFormula}
\end{equation}

We put
\begin{eqnarray*}
C_0&=&\limsup_{n\rightarrow\infty}C(P_n,1)^{1/k_n},\\
C_0^*&=&\liminf_{n\rightarrow\infty}C(P_n,1)^{1/k_n},\\
\eta (R)&=&\limsup _{{n\rightarrow\infty}}\frac{\eta (P_n,R)}{k_n}.
\end{eqnarray*}

For definition of capacity of a compact set, its properties and its relations to the Green's function and the harmonic measure of the set, one can refer to
\cite{fuc}.

For convenience of the reader we recall here Weierstrass's inequality (see Lemma 15.8 in \cite{rud})
\begin{lemma} For $p\in \mathbb N =\{1,2,\ldots \}$ define
\begin{eqnarray*}
G(z,p)=(1-z)\exp \{z+\frac{z^2}{2}+\ldots +\frac{z^p}{p}\}.
\end{eqnarray*}
(By convenience we define $G(z,0)=1-z$). Let $z\in \mathbb C$.

1. If $|z|\leq 1$ then
\begin{eqnarray*}
|G(z,p)|\leq |\exp \{\frac{z^{p+1}}{p+1}\}|(1+O(|z^{p+2}|)).
\end{eqnarray*}

2. If $z$ is small enough then
\begin{eqnarray*}
|G(z,p)|\sim |\exp \{\frac{z^{p+1}}{p+1}\}|(1+O(|z^{p+2}|)).
\end{eqnarray*}

3. If $|z|\geq 1$ then
\begin{eqnarray*}
|G(z,p)|\leq \exp \{\lambda _p|z|^p\}\leq \exp \{\lambda _p|z|^{p+1}\}
\end{eqnarray*}
where
\begin{eqnarray*}
\lambda _p=1+1+\frac{1}{2}+\frac{1}{3}+\ldots +\frac{1}{p}.
\end{eqnarray*}
Here big-oh $O(z)$ as usually means that there is a constant $C>0$ such that $O(z)\leq C|z|$ for $z$ small enough. \label{lem5}\end{lemma}
\begin{proof}
The proof of 3 is easy. For proof of 1 and 2 we apply Lemma 15.8 in \cite{rud} to the function $G(z,p+1)$.
\end{proof}

\begin{lemma} Let us assume that $k_n\geq d^*(P_n)$.

(i) If  $C_0=0$ then for all $R>0$
$$\limsup _{n\rightarrow\infty}||P_n||^{1/k_n}_{R}=0.$$

(ii) Assume that $C_0^*>0$ and
$$\limsup _{n\rightarrow\infty}C(P_n,R)^{1/k_n}\leq h(R),$$
where $h$ satisfies
$$\liminf _{R\rightarrow\infty}\frac{\log h(R)}{\log R}\leq \tau.$$
Then for all $R>0$ we have $\eta (R)\leq \tau.$
\label{lem2}\end{lemma}
\begin{proof}
See Lemma 1 in \cite{trong-tuyen2}.
\end{proof}
 We end this section with a result relating the maximum and logarithm norms
\begin{lemma} Let us assume that $k_n\geq d^*(P_n)$, $C_0<\infty$, and that
\begin{equation}
\lim_{R\rightarrow\infty}\limsup_{n\rightarrow\infty}\frac{\left |\sum_{|z_{n,j}|\geq R}\frac{1}{z_{n,j}}\right |}{k_n}=0,\label{lem3.1}
\end{equation}
and there exists a sequence $(R_n)$ of positive real numbers tending to $\infty$ such that
\begin{equation}
\limsup_{n\rightarrow\infty}\frac{\eta (P_n,R_n)}{k_n}<\infty.\label{lem3.2}
\end{equation}

If
$$\liminf_{R\rightarrow\infty}\limsup_{n\rightarrow\infty}\frac{\log C(P_n,R)^{1/k_n}}{\log R}\leq \tau,$$
then for all $R>0$ we have
$$\limsup_{n\rightarrow\infty}||P_n||_{R}^{1/k_n}\leq C_0(1+R)^{\tau}.$$
\label{lem3}\end{lemma}
\begin{proof}
See Lemma 2 in \cite{trong-tuyen2}.
\end{proof}
\section{The growth of sequences of entire functions of genus zero}
In this Section we still use the convenience in Section 2, that is $(P_n)$ is a sequence of entire functions of genus zero with representation
(\ref{RepresentationFormula}), and $(k_n)$ is a sequence of positive numbers.

Theorems \ref{theo1} and \ref{theo2} answer to Questions 2 and 3. Their proofs were given in \cite{trong-tuyen2}. In Theorem \ref{theo5} we answer
Question 3, whose results justify the conditions imposed in Theorems \ref{theo1} and \ref{theo2}. In Theorem \ref{theo6} we show how the conditions of
Theorems \ref{theo1} and \ref{theo2} can be improved in view of Theorem \ref{theo5}.
\begin{theorem}
Let $E$ be a closed set being non-thin at $\infty$. Let us assume that $k_n\geq d^*(P_n)$, and that
\begin{equation}
\lim_{R\rightarrow\infty}\limsup_{n\rightarrow\infty}\frac{|\sum_{|z_{n,j}|\geq R}1/z_{n,j}|}{k_n}=0,\label{theo1.1}
\end{equation}
and there exists a sequence $\{R_n\}$ of positive real numbers tending to $\infty$ such that
\begin{equation}
\limsup_{n\rightarrow\infty}\frac{\eta (P_n,R_n)}{k_n}<\infty.\label{theo1.2}
\end{equation}
If for each $z\in E$ one has
$$\limsup _{n\rightarrow\infty}|P_n(z)|^{1/k_n}\leq 1,$$
then
$$\limsup _{n\rightarrow\infty}||P_n||_R^{1/k_n}\leq 1,~\mbox{for all}~R>0.$$
\label{theo1}\end{theorem}
\begin{theorem}
Let $E$ be a closed set such that
\begin{equation}
\limsup _{R\rightarrow\infty}\frac{\log cap(E_R)}{\log R}=\beta >0.\label{theo2.1}
\end{equation}
Assume that (\ref{theo1.1}) and (\ref{theo1.2}) hold.  If for all $z\in E$ we have
$$\limsup _{n\rightarrow\infty}|P_n(z)|^{1/k_n}\leq h(|z|),$$
where
$$\limsup _{R\rightarrow\infty}\frac{\log h(R)}{\log R}\leq \gamma <\infty.$$
Then for all $R>0$ we have
$$\limsup _{n\rightarrow\infty}||P_n||^{1/k_n}_{R}\leq C_0(1+R)^{\gamma /\beta}.$$
\label{theo2}\end{theorem}
\begin{theorem}
Assume that $C_0^*>0$, and
\begin{eqnarray*}
\limsup _{n\rightarrow\infty}|P_n(z)|^{1/k_n}\leq 1
\end{eqnarray*}
for all $z\in\mathbb C$. We do not assume any more condition, in particular, we do not assume that $k_n\geq d^*(P_n)$.

Then we have the following two alternatives: Either

1) For any $m\in \mathbb N$
\begin{equation}
\lim _{R\rightarrow\infty}\limsup _{n\rightarrow\infty}\frac{1}{k_n}\sum _{|z_{n,j}|\geq R}\frac{1}{|z_{n,j}|^m}=\infty ,\label{theo5.1}
\end{equation}

or,

2) There exists $m\in \mathbb N$ such that
\begin{equation}
\lim _{R\rightarrow\infty}\limsup _{n\rightarrow\infty}\frac{1}{k_n}\sum _{|z_{n,j}|\geq R}\frac{1}{|z_{n,j}|^{m}}=0.\label{theo5.3}
\end{equation}
Moreover, if (\ref{theo5.3}) is the case, for any $l\in \mathbb N$
\begin{equation}
\lim _{R\rightarrow\infty}\limsup _{n\rightarrow\infty}\frac{1}{k_n}|\sum _{|z_{n,j}|\geq R}\frac{1}{z_{n,j}^l}|=0.\label{theo5.6}
\end{equation}

In particular, if $k_n\geq d^*(P_n)$ for all $n\in\mathbb N$ then alternative 2 holds, hence conditions (\ref{theo1.1}) and (\ref{theo1.2}) are
satisfied. \label{theo5}\end{theorem}
\begin{proof}
Assume that alternative 1 is not true. Then there exists $m\in \mathbb N$ such that
\begin{equation}
\lim _{R\rightarrow\infty}\limsup _{n\rightarrow\infty}\frac{1}{k_n}\sum _{|z_{n,j}|\geq R}\frac{1}{|z_{n,j}|^m}=0 .\label{theo5.4}
\end{equation}
From Lemma \ref{lem2}, it follows that
\begin{equation}
\limsup _{n\rightarrow\infty}C(P_n,R)\leq 1
\end{equation}
for all $R>0$ (see definition of $C(P_n,R)$ in Section 2).

Now using Lemma \ref{lem2} ii)  we can choose a sequence $R_n\rightarrow\infty $ such that
\begin{eqnarray*}
\limsup _{n\rightarrow\infty}\frac{\eta (P_n,R)}{k_n}=\kappa <\infty .
\end{eqnarray*}
(In fact we can take $\kappa =0$ under assumptions of Theorem \ref{theo5}. However here we present the proof in a way such that it can be directly
generated to the more general cases, see the remarks right after this proof).

Then define
$$Q_n(z)=a_nz^{\alpha _n}\prod _{|z_{n,j}|\leq R_n}(1-z/z_{n,j}),$$ and
\begin{eqnarray*}
H_n(z)=\prod _{|z_{n,j}|>R_n}(1-z/z_{n,j}).
\end{eqnarray*}

Then we have
\begin{eqnarray*}
\limsup _{n\rightarrow\infty}C(Q_n,R)^{1/k_n}=\limsup _{n\rightarrow\infty}C(P_n,R)^{1/k_n}\leq 1
\end{eqnarray*}
for all $R>0$. Then from Lemma \ref{lem3} we get
\begin{eqnarray*}
u(z)=\limsup _{n\rightarrow\infty}\frac{1}{k_n}\log |Q_n(z)|\leq \log C_0
\end{eqnarray*}
for all $z\in \mathbb C$. Since $C_0^*\geq 0$, for any $\theta \in [0,2\pi ]$ and for any $s>0$, since the set $\{Re^{i\theta}:~R\geq s\}$ is non-thin at
$\infty$, by Proposition \ref{lemmulyav} (or Theorem \ref{theo1}) there exists $R>s$ such that
\begin{equation}
\limsup _{n\rightarrow\infty}|Q_n(Re^{i\theta })|^{1/k_n}\geq C_0^*/2.\label{theo5.5}
\end{equation}
Use the following Taylor expansion (see also Lemma \ref{lem5})
\begin{eqnarray*}
\log (1-z)=-z-\frac{z^2}{2}-\frac{z^3}{3}-\ldots -\frac{z^{m-1}}{m}+O(|z|^m)
\end{eqnarray*}
for $|z|\leq 1/2$, by definition of $H_n(z)$ and \ref{theo5.3} we have
\begin{eqnarray*}
\liminf_{n\rightarrow\infty}|H_n(z)|^{1/k_n}=\liminf _{n\rightarrow\infty}| \exp\{-\frac{1}{k_n}[\frac{z}{1}\sum _{|z_{n,j}|\geq
R}\frac{1}{z_{n,j}}+\frac{z^2}{2}\sum _{|z_{n,j}|\geq R}\frac{1}{z_{n,j}^2}+\ldots +\frac{z^{m-1}}{m-1}\sum _{|z_{n,j}|\geq
R}\frac{1}{z_{n,j}^{m-1}}]\}|.
\end{eqnarray*}
Define
\begin{eqnarray*}
\beta _{n,l}=-\frac{1}{k_n}\sum _{|z_{n,j}|\geq R_n}\frac{1}{lz_{n,j}^{l}}=|\beta _{n,l}|e^{i\theta _{n,l}}
\end{eqnarray*}
for $l=1,2,\ldots $ where $\theta _{n,l}\in [0,2\pi ]$ is the argument of $\beta _{n.l}$. To prove (\ref{theo5.6}) it suffices to show that
\begin{eqnarray*}
\lim _{n\rightarrow\infty}|\beta _{n,l}|=0
\end{eqnarray*}
for all $l=1,2,\ldots ,m-1$ (for $l\geq m$ this claim is true in view of (\ref{theo5.3})). Assume by contradiction that the above claim is not true.
Passing to a subsequence we may assume then that
\begin{eqnarray*}
\lim _{n\rightarrow\infty}|\beta _{n,l}|&=&\beta _l,\\
\lim _{n\rightarrow\infty}|\theta _{n,l}|&=&\theta _l
\end{eqnarray*}
for all $l=1,2,\ldots ,m$ where we allow $\beta _l$ can take value $+\infty$, and at least one of $\beta _l$ is positive (we include here $\beta _m=0$
for convenience). Then passing to a further subsequence we may choose an $l\in \{1,2,\ldots ,m-1\}$ such that
\begin{eqnarray*}
\lim _{n\rightarrow\infty}\frac{|\beta _{n,l}|}{|\beta _{n,h}|}=+\infty
\end{eqnarray*}
for $h>l$, and
\begin{eqnarray*}
\lim _{n\rightarrow\infty}\frac{|\beta _{n,l}|}{|\beta _{n,h}|}>0
\end{eqnarray*}
for $h<l$. Fixed $\theta =-\theta _l/l$, choose $R>0$ such that (\ref{theo5.5}) is satisfied. For that $R$, by choosing of $l$, we have
\begin{eqnarray*}
\log (2C_0/C_0^*)&\geq&\liminf _{n\rightarrow\infty}[Re^{i\theta}\beta _{n,1}+R^2e^{2i\theta }\beta _{n,2}+\ldots +R^{m-1}e^{(m-1)i\theta}\beta _{n,m-1}]\\
&\geq& |R|^l|\beta _l|-|R|^{l-1}O(|\beta _l|).
\end{eqnarray*}
Since $R$ can be chosen as large as we like we conclude that $\beta _l=0$ which is a contradiction. This completes the proof of Theorem \ref{theo5}.
\end{proof}
The conditions (\ref{theo5.3}) and (\ref{theo5.6}) in the alternative 2) of Theorem \ref{theo5} turn out to be enough to achieve the conclusions of
Theorems \ref{theo1} and Theorem \ref{theo2}. So in some sense it is optimal.
\begin{theorem}
If in the statements of Theorem \ref{theo1} and Theorem \ref{theo2} we replace the conditions $k_n\geq d^*(P_n)$ and (\ref{theo1.1}) by the conditions
(\ref{theo5.3}) and (\ref{theo5.6}) in the alternative 2) of Theorem \ref{theo5} while keeping the other conditions, then their conclusions are still
true.
 \label{theo6}\end{theorem}
\begin{proof}
The proof is the same the proofs of Theorems \ref{theo1} and Theorem \ref{theo2}, using the Taylor's expansion of $\log (1-z)$ as in the proof of Theorem
\ref{theo5}.
\end{proof}
Remarks:

1. The results in this section can be generalized to any sequence $(P_n)$ of finite orders $\rho _n\leq \rho <\infty$ for all $n$ (see the discussions in
the last Section).

2. The proofs of Theorem \ref{theo5} can be appropriately changed to get similar results for the case we have the more growth
\begin{eqnarray*}
\limsup _{n\rightarrow\infty}|P_n(z)|^{1/k_n}\leq 1+|z|
\end{eqnarray*}
for all $z\in \mathbb C$.

3. In the alternative 2 of Theorem \ref{theo5} if we take into account the fact that we can choose $R_n\rightarrow\infty$ such that
\begin{eqnarray*}
\lim _{n\rightarrow\infty}\frac{\eta (P_n,R_n)}{k_n}=0,
\end{eqnarray*}
we can write (\ref{theo5.6}) as
\begin{eqnarray*}
\lim _{n\rightarrow\infty}\frac{1}{k_n}\sum _{|z_{n,j}|\geq 1}\frac{1}{z_{n,j}^m}=0.
\end{eqnarray*}

4. The LHS of (\ref{theo5.6}) can be effectively computed in practice. In fact, if
\begin{eqnarray*}
P(z)=a\prod _{j}(1-\frac{z}{z_j})
\end{eqnarray*}
then formally
\begin{eqnarray*}
\log P(z)=\log a-\frac{z}{1}\sum _{j}\frac{1}{z_j}-\frac{z^2}{2}\sum _{j}\frac{1}{z_j^2}-\frac{z^3}{3}\sum _{j}\frac{1}{z_j^3}-\ldots
\end{eqnarray*}
Hence we can compute
\begin{eqnarray*}
\sum _{j}\frac{1}{z_j}=-\frac{P'(0)}{P(0)}
\end{eqnarray*}
and so on.
\section{Some consequences and refinements}
Our first example of applications in this section is the following extension of Theorem 1 in \cite{mul-yav} (for the convenience of the reader, we state
the result in a similar manner to that of Theorem 1 in \cite{mul-yav}):
\begin{theorem}
Let $\Gamma$ be a continuum in $\Bbb C$ (for example, a continuous curve which is not a point), and let $E\subset \Bbb C$ be closed such that $E$ is
non-thin at $\infty$. Assume that $(P_n)$ is a sequence of entire functions of genus zero such that $d^*(P_n)\leq k_n$, where $(k_n)$ is an increasing
sequence of integer numbers such that conditions (\ref{theo1.1}) and (\ref{theo1.2}) are satisfied. Moreover, suppose that the following two conditions
are satisfied:

i) There is some function $f$ which is analytic on some simply connected open set $G_f\subset \Bbb C$ containing $\Gamma$ with
\begin{eqnarray*}
\limsup_{n\rightarrow\infty}||f-P_n||^{1/k_n}_{\Gamma}<1.
\end{eqnarray*}

ii) For all $z\in E$
\begin{eqnarray*}
\limsup _{n\rightarrow\infty}|P_n(z)|^{1/k_n}\leq 1.
\end{eqnarray*}
Then the following statements are true:

1) If $(k_{n+1}/k_n)$ is bounded, then $f$ extends to an entire function, and for any compact set $K$ of $\Bbb C$, we have
\begin{eqnarray*}
\limsup _{n\rightarrow\infty}||f-P_n||^{1/k_n}_K<1.
\end{eqnarray*}

2) If $(k_n)$ is arbitrary, then for any compact $K\subset G_f$ we have
\begin{eqnarray*}
\limsup _{n\rightarrow\infty}||f-P_n||^{1/k_n}_K<1.
\end{eqnarray*}
\label{ApplicationTheorem}\end{theorem} Remark: This theorem is an extension of Theorem 1 in \cite{mul-yav}: in case $(P_n)$ is a sequence of polynomials
and $k_n\geq deg(P_n)$, our result here is exactly that of Theorem 1 in \cite{mul-yav}. In fact, if we are in case 1) then as was shown in
\cite{mul-yav}, using Bernstein theorem we have that the condition i) in our Theorem is the same as the condition i) in Theorem \cite{mul-yav}. If we are
in case 2) then their condition and our condition are also the same.
\begin{proof} From the assumption, by Theorem \ref{theo1} we have
\begin{eqnarray*}
\limsup _{n\rightarrow\infty}||P_n||^{1/k_n}_K\leq 1
\end{eqnarray*}
for any compact set $K\subset \Bbb C$.

For proof of 2), we use Theorem 10 in \cite{wal}.

Now we prove 1). Again, by Theorem 10 in \cite{wal}, and assumption that $(k_{n+1}/k_n)$ is bounded, there is a non-empty open set $U$ contained in $G_f$
for which
\begin{eqnarray*}
\limsup _{n\rightarrow\infty}||P_{n+1}-P_n||^{1/k_n}_U<1.
\end{eqnarray*}
Then use the property at the beginning of this proof, argue similarly to the proof of Theorem 1 in \cite{mul-yav} (that is, use two-constants theorem) we
get that
\begin{eqnarray*}
\limsup _{n\rightarrow\infty}||P_{n+1}-P_n||^{1/k_n}_K<1
\end{eqnarray*}
for any compact set $K\subset \Bbb C$. From this, it is straightforward to imply that $f$ extends to an entire function and
\begin{eqnarray*}
\limsup _{n\rightarrow\infty}||f-P_n||^{1/k_n}_K<1
\end{eqnarray*}
for any compact set $K\subset \Bbb C$.
\end{proof}
Now we come to another interesting corollary, which gives a criterion for local convergence of a summation of a sequence of entire functions of genus
zero
\begin{theorem} Assume that $(P_n)$ is a sequence of entire functions of genus zero, $(k_n)$ is an increasing sequence of integer numbers, and $E$ is a closed
subset of $\Bbb C$ such that conditions of Theorem \ref{theo2} are satisfied. Moreover assume that there exist $z_0\in \Bbb C$ and $\rho >0$ such that
\begin{eqnarray*}
C:=\limsup _{n\rightarrow\infty}\exp\{\frac{1}{2\pi k_n}\int _0^{2\pi}\log |P_n(z_0+\rho \theta )|d\theta\}<1.
\end{eqnarray*}
Let $\beta >0,\gamma$ denote the constants in the statement of Theorem \ref{theo2}. Then $\sum _{n=1}^{\infty}P_n$ converges uniformly in the set
$\{|z-z_0|<\rho (\frac{1}{C^{\beta /\gamma }}-1)\}$. In particular, if $\gamma =0$ then $\sum _{n=1}^{\infty}P_n$ converges uniformly in $\Bbb C$.
\label{MainApplicationTheorem}\end{theorem}
\begin{proof}
Define $Q_n(z)=P_n(z_0+\rho z)$. Because we will apply Theorem \ref{theo2} for $P_n$, we check here that the conditions of Theorem \ref{theo2} is
satisfied for sequences $(Q_n)$ and $(k_n)$. First not that the condition $k_n\geq d^*(Q_n)$ is in fact not needed in the proof of all previous results,
we need only that
\begin{eqnarray*}
\limsup _{n\rightarrow\infty}\frac{d^*(Q_n)}{k_n}<\infty ,
\end{eqnarray*}
and this condition is obviously satisfied from our assumptions on $P$. We can also choose $R_n's$ for $Q$ such that the other conditions of Theorem
\ref{theo2} are satisfied.

From the above analysis, without loss of generality, we may assume that $z_0=0,~\rho =1$.

Then $C=C_0$, where $C_0$ is the constant defined in the beginning of Section 2. By Theorem \ref{theo2}, we have
\begin{eqnarray*}
\limsup _{n\rightarrow\infty}||P_n||^{1/k_n}_R\leq C_0(1+R)^{\gamma /\beta}
\end{eqnarray*}
for any $R>0$. Then the conclusions of Theorem \ref{MainApplicationTheorem} is straightforward to obtain.
\end{proof}
Theorem \ref{MainApplicationTheorem} can be applied to such sequences as $P_n(z)=z^n$ and $k_n=n$. In fact, the conditions of Theorem
\ref{MainApplicationTheorem} are natural: If a sequence of polynomials $P_n(z)$ is bounded in any nonempty open set of $\mathbb C$ then
\begin{eqnarray*}
|P_n(z)|^{1/deg(P_n)}\leq C(1+|z|)
\end{eqnarray*}
for all $z\in\mathbb C$ where $C$ is a constant.

We finish this section by outlining some refinements (see also Theorems \ref{theo5} and \ref{theo6}):

Firstly, in theory we can effectively choose a sequence $(k_n)$ satisfying condition (\ref{theo1.2}), given a sequence of entire function of genus zero
$(P_n)$. To this end, we note that if $P(z)$ is an entire function of genus zero, then Lemma 2 in \cite{lev} shows that
\begin{eqnarray*}
\rho _1(P):=\limsup _{r\rightarrow\infty}\frac{\log \eta (P,r)}{\log r}\leq 1.
\end{eqnarray*}
Here $\rho _1(P)$ is called the order of counting function for $P(z)$. (It is known, see Theorem 2 page 18 in \cite{lev}, that $\rho _1(P)$  does not
exceed the growth order $\rho (P)$ that we defined in Section 2.) Hence condition (\ref{theo1.2}) is satisfied if we choose $k_n$ as
\begin{eqnarray*}
k_n\approx R_n^{\rho (P_n)}\lesssim R_n.
\end{eqnarray*}

Lastly, condition (\ref{theo1.1}) is trivially satisfied if the functions $P_n(z)$ have some symmetric properties. For example, this is the case if
$P_n(z)$ has the form
\begin{eqnarray*}
P_n(z)=\prod _{m}P_{n,m}(z)
\end{eqnarray*}
where $P_{n,m}$ are polynomials whose sums of the reverses of zeros (roots) are zero identically, and such that the absolute value of any zero of
$P_{n,m+1}$ is greater than the absolute value of any zero of $P_{n,m}$, for any $m$. It was observed in \cite{lev} (page 32) that this symmetric
property affects the growth of an entire function.

\section{Examples}
Example 1: grouped power series.

A grouped power series which can be formally written as
\begin{equation}
s(z)=\sum _{n=0}^{\infty}z^{k_n}P_{n}(z)\label{Sec6.1}
\end{equation}
where $(k_n)$ be an increasing sequence of positive integers, and $P_n(z)$ are entire functions. In \cite{mul-yav} they considered the case where
$P_n(z)$ are polynomials of degree $\mu _k$ such that $\lambda _k+\mu _k<\lambda _{k+1}$. Here we consider a more general case that $P_n(z)$ are entire
functions of genus zero.

The following result is a generalization of Corollary page 199 in \cite{mul-yav}
\begin{proposition}
Let $E$ be a subset of $\Bbb C$ which is non-thin at $\infty$. Let $s(z)$ be defined as in (\ref{Sec6.1}) where $(P_n)$ is a sequence of entire functions
of genus zero, and where $(k_n)$ is an increasing sequence of positive integers. Assume more that $k_n\geq d^*(P_n)$ for all $n$, and conditions
(\ref{theo1.1}) and (\ref{theo1.2}) are satisfied.

If $s(z)$ converges pointwise on $E$ then it converges uniformly on compact sets of $\Bbb C$. \label{GroupedPowerSeries}\end{proposition}
\begin{proof}
Define
\begin{eqnarray*}
Q_n(z)=z^{k_n}P_n(z).
\end{eqnarray*}
Then the sequences $(Q_n)$ and $(2k_n)$ satisfy conditions (\ref{theo1.1}) and (\ref{theo1.2}), and $2k_n\geq d^*(Q_n)$ for all $n$.

Now since $s(z)$ converges pointwise on $E$, we have
\begin{eqnarray*}
Q_n(z)=s_{n}(z)-s_{n-1}(z)
\end{eqnarray*}
converges pointwise to $0$ on $E$. In particular,
\begin{eqnarray*}
\limsup _{n\rightarrow\infty}|Q_n(z)|^{1/(2k_n)}\leq 1
\end{eqnarray*}
for all $z\in E$. Apply Theorem \ref{theo1} we get
\begin{eqnarray*}
\limsup _{n\rightarrow\infty}||Q_n||_R^{1/(2k_n)}\leq 1
\end{eqnarray*}
for all $R>0$. From the definition of $Q_n$ we have that
\begin{eqnarray*}
\limsup _{n\rightarrow\infty}||P_n||_R^{1/(2k_n)}\leq R^{-1/2}
\end{eqnarray*}
for all $R>0$. Now using that the LHS of above inequality is not decreasing as a function of $R$, we have that
\begin{eqnarray*}
\limsup _{n\rightarrow\infty}||P_n||_R^{1/(2k_n)}=0
\end{eqnarray*}
for all $R>0$. Then we also have
\begin{eqnarray*}
\limsup _{n\rightarrow\infty}||Q_n||_R^{1/(2k_n)}=0
\end{eqnarray*}
for all $R>0$. From this it follows easily that $s(z)$ converges uniformly in every compact subset of $\Bbb C$.
\end{proof}
Example 2: Fourier transform with real kernels of compact support.

Let $\varphi$ be a nontrivial function of $L^1(\RR )$. Then its Fourier transform is defined by
\begin{eqnarray*}
\widehat{\varphi}(x)=\int _{-\infty}^{\infty}e^{-ixt}\varphi (t)dt,~x\in \RR .
\end{eqnarray*}
If $\varphi$ is of compact support, there associates an entire function of exponential type of Laplace transform type
\begin{equation}
\Phi (z)=\int _{-\infty}^{\infty}e^{zt}\varphi (t)dt.\label{fou}
\end{equation}
From Lemma 3 in \cite{trong-tuyen1}, if $\varphi (t)=0$ for all $t\leq 0$ then $\Phi (z)$ has a representation
\begin{equation}
\Phi (z)=Ce^{z\frac{\sigma +\mu }{2}}z^{\alpha }\prod _{z_{j}\in \mathbb R}(1-\frac{z}{z_{j}})\prod
_{Im~z_{j}>0}(1-\frac{z}{z_{j}})(1-\frac{z}{\overline{z_{j}}}),\label{Sec7.0}
\end{equation}
where
\begin{eqnarray*}
\sigma &=&\sup \{a>0: \varphi |_{[0,a]}=0~a.e.\},\\
\mu &=&\inf \{a\in [0,\infty ):\varphi |_{[a,\infty )}=0~a.e.\}.
\end{eqnarray*}
Moreover
\begin{eqnarray*}
\sum _{j}|Re(\frac{1}{z_{j}})|&<&\infty ,\\
\sum _{j}\frac{1}{|z_{j}|^q}&<&\infty
\end{eqnarray*}
for all $q>1$. From (\ref{Sec7.0}) we can explore the growth of a sequence of entire functions $(\Phi _n)$ which are Laplace transforms of real kernel
with compact support in a manner that is similar to that of Sections 2, 3, and 4, under appropriate changes on the assumptions  (See Theorem
\ref{FourierTransformationTheorem} for one way to doing that). For example, we may define the degree of an entire function $\Phi (z)$ having the
representation (\ref{Sec7.0}) by
\begin{eqnarray*}
d^*(\Phi )=\sigma +\mu +\alpha +\sum _{j}|Re~\frac{1}{z_j}|+\sum _{j}\frac{1}{|z_j|^q}
\end{eqnarray*}
for some fixed $1<q\leq 2$. However we will not pursuit this issue in this section (see also discussions in the last Section).

Now we go back to the Laplace transform (\ref{fou}). Generally, if $\varphi\in L^1(\RR )$  satisfies
\begin{equation}
\int _{-\infty}^{\infty}e^{|t|s}|\varphi (t)|dt <\infty , \label{150506}
\end{equation}
for all $s>0$ then $\Phi$ in (\ref{fou}) is defined on all over the complex plane $\CC$ and is an entire function, but its order may be any positive
number. If $\varphi \geq 0$ then condition (\ref{150506}) is also necessary for $\Phi$ to be an entire function. Indeed, in this case we have, for every
$R>0$,
\begin{eqnarray*}
\max _{|z|\leq R}|\Phi (z)|\geq \frac{1}{2}(\Phi (R)+\Phi (-R))\geq \frac{1}{2}\int _{-\infty}^{\infty}e^{|t|R}\varphi (t)dt.
\end{eqnarray*}
Condition (\ref{150506}) is rather strict. We may ask a more general question: For what function $\varphi$ being nontrivial and in $L_{loc}^1(\mathbb R)$
(meaning $\varphi \in L^1(K)$ for all compact set $K$ of $\mathbb R$) that the sequence
\begin{equation}
\Phi _n(z)=\int _{-n}^ne^{zt}\varphi (t)dt\label{Sec7.1}
\end{equation}
converges locally uniformly in $\mathbb C$ to an entire nonzero function $\Phi (z)$? Applying the results in previous sections, we will show that there
is an obstruction to this question, which can be presented in terms of zeros of $\Phi _n(z)$ (see Theorem \ref{FourierTransformationTheorem}).

For simplicity we will assume that $\varphi (t)=0$ for $t\leq 0$. From formula (\ref{Sec7.0}), for $n$ large enough, $\Phi _n$ has a representation
\begin{equation}
\Phi _n(z)=C_ne^{z\frac{\sigma +\mu _n}{2}}z^{\alpha _n}\prod _{z_{j,n}\in \mathbb R}(1-\frac{z}{z_{j,n}})\prod
_{Im~z_{j,n}>0}(1-\frac{z}{z_{j,n}})(1-\frac{z}{\overline{z_{j,n}}}),\label{Sec7.2}
\end{equation}
where
\begin{eqnarray*}
\sigma &=&\sup \{a>0: \varphi |_{[0,a]}=0~a.e.\},\\
\mu _n&=&\inf \{a\in [0,n]:\varphi |_{[a,n]}=0~a.e.\}.
\end{eqnarray*}
Moreover
\begin{eqnarray*}
\sum _{j}|Re(\frac{1}{z_{j,n}})|&<&\infty ,\\
\sum _{j}\frac{1}{|z_{j,n}|^q}&<&\infty
\end{eqnarray*}
for all $q>1$.

Then we have the following result
\begin{theorem}
Given $\varphi \in L^1_{loc}(\mathbb R)$, $\varphi (t)=0$ for $t\leq 0$, and the support of $\varphi$ is noncompact. Let $\Phi _n$ be defined as above,
and let (\ref{Sec7.2}) be its representation. Let $\sigma$ and $\mu _n$ be defined as above. Assume that the following three conditions are true

1.
\begin{eqnarray*}
\lim _{R\rightarrow\infty}\limsup _{n\rightarrow\infty}\frac{1}{\mu _n}\sum _{|z_{j,n}|\geq R}|Re(\frac{1}{z_{j,n}})|<\infty .
\end{eqnarray*}

2.\begin{eqnarray*} \lim _{R\rightarrow\infty}\limsup _{n\rightarrow\infty}\frac{1}{\mu _n}|\sum _{|z_{j,n}|\geq R}\frac{1}{z_{j,n}}|=0.
\end{eqnarray*}

3. There exists $1<q<2$ such that
\begin{eqnarray*}
\lim _{R\rightarrow\infty}\limsup _{n\rightarrow\infty}\frac{1}{\mu _n}\sum _{|z_{j,n}|\geq R}\frac{1}{|z_{j,n}|^q}<\infty .
\end{eqnarray*}

Then there is no nonzero entire function $\Phi (z)$ such that $\Phi _n(z)$ converges locally uniformly in $\mathbb C$ to $\Phi (z)$.
\label{FourierTransformationTheorem}\end{theorem}
\begin{proof}
We prove by contradiction. Assume by contradiction that $\Phi _n(z)$ converges locally uniformly in $\mathbb C$ to a nonzero entire function $\Phi (z)$.

Since the support of $\varphi$ is noncompact, we have
\begin{equation}
\lim _{n\rightarrow\infty}\mu _n =\infty .\label{FourierTransformationTheorem.1}
\end{equation}
Hence
\begin{eqnarray*}
\lim _{n\rightarrow\infty}|P_n(z)|^{1/\mu _n}\leq 1,
\end{eqnarray*}
for all $z\in \mathbb C$, and the equality is true at every $z\mathbb C$ with $\Phi (z)\not= 0$. Using (\ref{FourierTransformationTheorem.1}) and
Rouche's theorem (applied for the sequence $P_n$ converging to $\Phi$) we can find a sequence $R_n\rightarrow\infty$ such that
\begin{eqnarray*}
\limsup _{n\rightarrow\infty}\frac{\eta (P_n,R_n)}{\mu _n}\leq 1.
\end{eqnarray*}
Now define
\begin{eqnarray*}
Q_n(z)=C_nz^{\alpha _n}\prod _{|z_{j,n}|\leq R_n}(1-\frac{z}{z_{j,n}}).
\end{eqnarray*}
Use conditions 1, 2, 3 of Theorem \ref{FourierTransformationTheorem}, arguing similar to the proofs of results in Sections 2 and 3, for every $z\in
\mathbb C$
\begin{eqnarray*}
\lim _{n\rightarrow\infty}|e^{z/2}|.|Q_n(z)|^{1/\mu _n}\leq 1.
\end{eqnarray*}
Fix $R>0$. If $z\in \mathbb R$ and $z\geq 2R$ then
\begin{eqnarray*}
\lim _{n\rightarrow\infty}|Q_n(z)|^{1/\mu _n}\leq e^{-R}.
\end{eqnarray*}
Since the set $\{z\in \mathbb R:~z\geq R\}$ is non-thin at $\infty$, from the above inequality we have by Theorem \ref{theo1}
\begin{eqnarray*}
\lim _{n\rightarrow\infty}|Q_n(z)|^{1/\mu _n}\leq e^{-R}
\end{eqnarray*}
for all $z\in \mathbb C$ and for all $R>0$. Since $R>0$ is arbitrary we have from the above that
\begin{eqnarray*}
\lim _{n\rightarrow\infty}|Q_n(z)|^{1/\mu _n}=0
\end{eqnarray*}
for all $z\in \mathbb C$. Then we have
\begin{eqnarray*}
\lim _{n\rightarrow\infty}|\Phi _n(z)|^{1/k_n}=0
\end{eqnarray*}
for all $z\in \mathbb C$, which is a contradiction.
\end{proof}
Refer to Theorem \ref{theo5} we note that the LHS of (\ref{theo5.6}) can be computed via the integrals of $\varphi (t)$. For example we have
\begin{proposition}
Let $\varphi \in L^1(\mathbb R)$ of compact support and $\varphi (t)=0$ for all $t\leq 0$. Define $\Phi (z)$ its Laplace transform by formula (\ref{fou})
with representation (\ref{Sec7.0}). Assume that $\Phi (0)\not= 0$. Then
\begin{equation}
\sum _{j}\frac{1}{z_j}=\frac{\sigma +\mu }{2}-\frac{\Phi '(0)}{\Phi (0)}=\frac{\sigma +\mu }{2}-\frac{\int _{-\infty}^{\infty}t\varphi (t)dt}{\int
_{-\infty}^{\infty}\varphi (t)dt}\label{lem4.1}.
\end{equation}
\label{lem4}\end{proposition}
\begin{proof}
Formally we have
\begin{equation}
\log \Phi (z)=\log C+\frac{\sigma +\mu }{2}z+\sum _{z_j\in\mathbb R}\log (1-\frac{z}{z_j})+\sum _{Im ~z_j>0}\log
[1-z(\frac{1}{z_j}+\frac{1}{\overline{z_j}})+\frac{z^2}{|z_j|^2}].\label{lem4.3}
\end{equation}
Then compute
\begin{eqnarray*}
\frac{d}{dz}\log \Phi (z)|_{z=0}
\end{eqnarray*}
using formulas (\ref{fou}) and (\ref{lem4.3}) we get (\ref{lem4.1}).
\end{proof}
Example 3: grouped Fourier series.

Similarly to above Examples, we can consider series of the form
\begin{eqnarray*}
s(z)=\sum _{j}z^{k_n}P_n(z)\sin (\lambda _nz)
\end{eqnarray*}
where $k_n\geq n$, $P_n$ is a polynomial of degree $deg(P_n)\leq k_n$, and $\lambda _n\in \mathbb C$ which satisfies
\begin{eqnarray*}
\lim _{n\rightarrow\infty}\lambda _n/k_n=0.
\end{eqnarray*}
Claim: If $P_n(-z)=-P_n(z)$ for all $z\in \mathbb C$ and $n\in \mathbb N$ (or $P_n(-z)=P_n(z)$ for all $z\in \mathbb C$ and $n\in \mathbb N$), and $s(z)$
converges on a set $E$ such that the set $\{z^2:~z\in E\}$ is non-thin at $\infty$ then $s(z)$ converges locally uniformly on $\mathbb C$.

This Claim can be proved in the same spirit as that of Lemma \ref{GroupedPowerSeries}, using the factorization of the function $\sin z$.
\section{Conclusions}
1. The growth of sequences of finite orders

We can develop the parallelism of results in Sections 2, 3, 4 and 5 to this general class of entire functions. We will not go into details but will
outline here how we can do for the general case.

i) Theorem \ref{theo5} can be proved in this generality with only a little change in the conclusion.

ii) Now we show how we can generalize Theorems \ref{theo1}.

We will assume that $(P_n)$ is a sequence of entire functions of finite orders $\rho _n\leq \rho$ where $\rho >0$ is a constant, with representation
\begin{equation}
P_n(z)=a_nz^{m_n}e^{W_{q_n}(z)}\prod _{j}G(\frac{z}{z_{n,j}},p_n)\label{Sec8.3}
\end{equation}
where $p_n=[\rho _n]$ the integer part of $\rho_n$, $W_{q_n}(z)$ is a polynomial of degree $q_n\leq \rho _n$ with $W_{q_n}(0)=0$, $z_{n,j}'s$ are zeros
of $P(z)$ different from $0$. Let us assume that $(k_n)$ is a sequence of positive numbers with $k_n\geq d^*(P_n)$ for all $n\in \mathbb N$. A less
strong condition on $k_n$ can also be used, however the presentation of the assumptions will be very complicated so we will skip it here.

We also assume the condition (\ref{theo1.2}), while condition (\ref{theo1.1}) is replaced by
\begin{equation}
\lim _{R\rightarrow\infty}\limsup _{n\rightarrow\infty}\frac{1}{k_n}|\sum _{|z_{n,j}|> R}\frac{1}{z_{n,j}^{p_n+1}}|=0.
\end{equation}
Under these conditions we have
\begin{eqnarray*}
\limsup _{n\rightarrow\infty}|P_n(z)|^{1/k_n}=\limsup _{n\rightarrow\infty}|Q_n(z)|^{1/k_n}
\end{eqnarray*}
for all $z\in \mathbb C$, where
\begin{equation}
Q_n(z)=\exp\{W_{q_n}(z)+\sum _{|z_{n,j}|\leq R_n}(\frac{z}{z_{n,j}}+\frac{z^2}{2z^2_{n,j}}+\ldots +\frac{z^{p_n}}{p_nz^{p_n}_{n,j}})\}\times
a_nz^{m_n}\prod _{|z_{n,j}|\leq R_n}(1-\frac{z}{z_{n_j}}).\label{Sec8.4}
\end{equation}
Also, from the assumptions we have
\begin{eqnarray*}
u(z)=\limsup _{n\rightarrow\infty}\frac{1}{k_n}\log |a_nz^{m_n}\prod _{|z_{n,j}|\leq R_n}(1-\frac{z}{z_{n_j}})|\leq \kappa \log ^+|z|+C
\end{eqnarray*}
for all $a\in\mathbb C$ for some constants $\kappa ,C>0$, and all polynomials
\begin{eqnarray*}
\frac{1}{k_n}[W_{q_n}(z)+\sum _{|z_{n,j}|\leq R_n}(\frac{z}{z_{n,j}}+\frac{z^2}{2z^2_{n,j}}+\ldots +\frac{z^{p_n}}{p_nz^{p_n}_{n,j}})]
\end{eqnarray*}
are polynomials of degree less than $\rho +1$, and have bounded coefficients. Passing to subsequences we may assume that
\begin{eqnarray*}
\lim _{n\rightarrow\infty}\frac{1}{k_n}[W_{q_n}(z)+\sum _{|z_{n,j}|\leq R_n}(\frac{z}{z_{n,j}}+\frac{z^2}{2z^2_{n,j}}+\ldots
+\frac{z^{p_n}}{p_nz^{p_n}_{n,j}})]=W(z)
\end{eqnarray*}
exists for all $z\in \mathbb C$ and is also a polynomial of degree $deg(W)\leq \rho +1$. Then
\begin{equation}
\limsup _{n\rightarrow\infty}\frac{1}{k_n}\log |P_n(z)|=Re(W(z))+u(z).\label{Sec8.5}
\end{equation}
From the above analysis, it is clear to see that in order to extend Theorem \ref{theo1} we need the condition on the set $E$ such that if
\begin{eqnarray*}
Re(W(z))+u(z)\leq 0
\end{eqnarray*}
for all $z\in E$, where $W(z)$ is a polynomial of degree $deg(W)\leq \rho +1$, $W(0)=0$ and $u(z)$ is a subharmonic function with growth bounded from
above by $\kappa \log ^+|z|+C$ then
\begin{eqnarray*}
Re(W(z))+u(z)\leq 0
\end{eqnarray*}
for all $z\in\mathbb C$. We do not know what is a sufficient and necessary conditions for sets $E$ having this property (we will discuss more about this
point in the second part). Here we give a sufficient condition for that property:
\begin{lemma} Let $E\subset \mathbb C$ be a collection of real lines $L$ going
through $0$ such that for all $(1+2k)!$-th roots $\zeta$ of unity (meaning $\zeta ^{(1+2k)!}=1$) and for all line $L\subset E$, then $\zeta L=\{\zeta
z:~z\in L\}\subset E$ (here $k!=1.2\ldots k$ is the factorial of $k$). Let $W(z)$ be a polynomial of degree $deg(W)\leq k$, $W(0)=0$. Let $u(z)$ be a
subharmonic function with growth bounded from above by $\kappa \log ^+|z|+C$. If
\begin{eqnarray*}
Re(W(z))+u(z)\leq 0
\end{eqnarray*}
for all $z\in E$ then
\begin{eqnarray*}
Re(W(z))+u(z)\leq 0
\end{eqnarray*}
for all $z\in\mathbb C$.
\end{lemma}
\begin{proof}
If $W\equiv 0$ then since $E$ is non-thin at $\infty$ we are done.

Assume that $W\not\equiv 0$. Then $1\leq deg(W)=l\leq k$. Since the number of common zeros of the system
\begin{eqnarray*}
Re(z^l)&=&0,\\
|z|&=&1
\end{eqnarray*}
does not greater than $2l$, from the assumptions it is easy to see that there exists at least one line $L\subset E$ such that
\begin{eqnarray*}
\lim _{z\in L,~|z|\rightarrow\infty}Re(W(z))=+\infty .
\end{eqnarray*}
Then we even get a more stronger conclusion $u(z)\equiv -\infty$.
\end{proof}

iii) Theorem \ref{theo2} can also be generalized in the same manner. We skip the details here.

2. Extremal functions

Let $E$ be a subset of $\mathbb C$. The discussion in part 1) shows that it is natural to consider the following function
\begin{equation}
V_{E,q}(z)=\sup \{U(z):~U\in \mathbb L_q,~U|_E\leq 0\}\label{Sec8.6}
\end{equation}
where $\mathbb L_q$ is the set of all functions $U(z):\mathbb C\rightarrow [-\infty ,+\infty )$ having the form
\begin{eqnarray*}
U(z)=Re(W(z))+u(z)
\end{eqnarray*}
where $W(z)$ is a polynomial of degree $deg(W)\leq q$ and $W(0)=0$, and where $u(z)\leq \log ^+|z|+C$ is a subharmonic function where $C$ is a constant
depending on $u$.

The function $V_{E,q}(z)$ in (\ref{Sec8.6}) is an analogous of the Siciak's extremal function (see \cite{klimek}). Similar functions were also considered
in \cite{bos-brudnyi-levenberg-totik}. We hope to return to this in a future paper.

\end{document}